\documentclass[11pt]{article}

\usepackage{amsfonts,amsmath,latexsym,color,epsfig}
\newtheorem{theorem}{Theorem}
\newtheorem{lemma}{Lemma}
\newtheorem{proposition}{Proposition}

\newtheorem{corollary}{Corollary}

\newtheorem{claim}{Claim}

\newtheorem{definition}{Definition}

\newenvironment{proof}
      {\medskip\noindent{\bf Proof:}\hspace{1mm}}
      {\hfill$\Box$\medskip}

\def\qed{\ifvmode\mbox{ }\else\unskip\fi\hskip 1em plus 10fill$\Box$}
\input{epsf}
\makeatletter
\def\Ddots{\mathinner{\mkern1mu\raise\p@
\vbox{\kern7\p@\hbox{.}}\mkern2mu
\raise4\p@\hbox{.}\mkern2mu\raise7\p@\hbox{.}\mkern1mu}}
\makeatother

\title{\vspace{-0.7cm}Decompositions into subgraphs of small diameter}
\author{Jacob
Fox\thanks{Department of Mathematics, Princeton, Princeton, NJ.
Email: {\tt jacobfox@math.princeton.edu}. Research supported by an
NSF Graduate Research Fellowship and a Princeton Centennial
Fellowship.} \and Benny Sudakov\thanks{Department of Mathematics,
UCLA,  Los Angeles, CA 90095. Email: {\tt bsudakov@math.ucla.edu}. Research supported in part by NSF CAREER award DMS-0812005 and by
USA-Israeli BSF grant.}}
\date{}
\begin{document}
\maketitle

\begin{abstract}

We investigate decompositions of a graph into a small number of low diameter subgraphs. Let $P(n,\epsilon,d)$ be
the smallest $k$ such that every graph $G=(V,E)$ on $n$ vertices has an edge partition $E=E_0
\cup E_1 \cup \ldots \cup E_k$ such that $|E_0| \leq \epsilon n^2$ and for all $1 \leq i \leq k$ the diameter of
the subgraph spanned by $E_i$ is at most $d$. Using Szemer\'edi's
regularity lemma, Polcyn and Ruci\'nski showed that $P(n,\epsilon,4)$ is bounded above by a
constant depending only $\epsilon$. This shows that every dense graph can be partitioned into a small number of
``small worlds'' provided that few edges can be ignored. Improving on
their result, we determine $P(n,\epsilon,d)$ within an absolute constant factor, showing that $P(n,\epsilon,2)=
\Theta(n)$ is unbounded for $\epsilon< 1/4$, $P(n,\epsilon,3) =\Theta(1/\epsilon^2)$ for
$\epsilon> n^{-1/2}$ and $P(n,\epsilon,4) =\Theta(1/\epsilon)$ for $\epsilon>n^{-1}$. We also prove that if $G$
has large minimum degree, {\em all the edges} of $G$ can be covered by a small number of low
diameter subgraphs. Finally, we extend some of these results to hypergraphs, improving earlier work of Polcyn,
R\"odl, Ruci\'nski, and Szemer\'edi. \end{abstract}

\section{Introduction}
The {\it distance} between two vertices of a graph is the length of the shortest path between them. The {\it
diameter} $\textrm{diam}(G)$ of a connected graph $G=(V,E)$ is the maximum distance between any
pair of vertices of the graph. If $G$ is not connected, $\textrm{diam}(G)=\infty$. For an edge subset $E'
\subset E$, the diameter of $E'$ is the diameter of the subgraph of $G$ with edge set $E'$, whose
vertex set consists of all the vertices of $G$ which belong to at least one edge of $E'$.

Extremal problems on the diameter of graphs have a long history and were first investigated by Erd\H{o}s and
R\'enyi \cite{ErRe} and Erd\H{o}s, R\'enyi, and S\'os \cite{ErReSo}, who studied the minimum number of edges in an
$n$-vertex graph with diameter at most $d$. Another line of research concerning the change of diameter if edges of the graph
are added or deleted was initiated by Chung and Garey \cite{CG} (see also, e.g.,
\cite{EGR, AGR}). In this paper, we investigate decompositions of a graph
into a small number of subgraphs of low diameter. A motivation for such decompositions comes from distributed
computing. We are given a set of processors (vertices) with communication channels
(edges) between pairs of processors. A fundamental problem when designing algorithms on such systems is
determining how much coordination must be done between the processors and accomplishing this
coordination as efficiently as possible. The simplest approach is to centralize the network by appointing one
processor to coordinate the actions of the network. This approach often simplifies the problem
and leads to distributed algorithms based on known serial algorithms. However, if the network has large diameter,
such rigid centralization can degrade system performance due to delays in communication.
One solution to this problem is to partition the network into regions
of low diameter. This approach was used for example by Linial and Saks
\cite{LiSa} who showed that every graph on $n$ vertices
can be vertex partitioned into $O(\log n)$ induced subgraphs whose connected
components have diameter $O(\log n)$.

In this paper we will consider another variant of the low diameter decomposition problem, which was also studied by several
researchers. In this problem the goal is to partition nearly all the edges of a graph into a small
number of low diameter subgraphs. More precisely we study the following parameter.

\begin{definition}
Let $P(n,\epsilon,d)$ be the smallest $\ell$ such that every graph $G=(V,E)$ on $n$ vertices has an edge partition
$E=E_0 \cup E_1 \cup \ldots \cup E_{\ell}$ such that $|E_0| \leq \epsilon n^2$ and the diameter of $E_i$ is at most
$d$ for $1 \leq i \leq \ell$.
\end{definition}

Polcyn and Ruci\'nski \cite{PoRu} recently showed that every dense graph can be partitioned into a small
number of ``small worlds'' provided that a small fraction of the edges can be ignored. Specifically, they proved that
$P(n,\epsilon,4)$ is bounded by a constant depending only on $\epsilon$. Their proof relies on Szemer\'edi's regularity
lemma and consequently gives an enormous upper bound on $P(n,\epsilon,4)$ as a function of $\epsilon$, i.e.,
it shows that $P(n,\epsilon,4)$ can be bounded from above by a tower of $2$s of height polynomial in $1/\epsilon$.
One of our main results determines $P(n,\epsilon,d)$ up to a constant factor. It improves on the result of Polcyn and Ruci\'nski
both on the diameter bound and on the number of parts.

\begin{theorem}\label{main1} (a) For $\epsilon<1/4$ bounded away from $1/4$, we have $P(n,\epsilon,2)=\Theta(n)$. \\
(b) For $\epsilon \geq n^{-1/2}$, we have $P(n,\epsilon,3)=\Theta(1/\epsilon^2)$. \\
(c) For $\epsilon \geq n^{-1}$, we have $P(n,\epsilon,4)=\Theta(1/\epsilon)$.
\end{theorem}

There is a sharp transition in the behavior of the function $P(n,\epsilon,2)$ at $\epsilon=1/4$,
namely $P(n,1/4,2)$$=1$. Note also that $P(n,\epsilon,1)$ is the minimum number of edge-disjoint cliques needed to cover
all but $\epsilon n^2$ edges in any graph on $n$ vertices. This parameter is not so interesting to study as for $\epsilon$ not too
large,
$P(n,\epsilon,1)$ is quadratic in $n$ by considering the complete bipartite graph with parts of equal size.

Extremal problems on the diameter of graphs with large minimum degree have also been studied. For example, Erd\H{o}s
et al.~\cite{ErPaPoTu}, answering a question of Gallai, determine up to an additive constant the largest possible
diameter of a connected graph with a given number $n$ of vertices and minimum degree $\delta$. The answer is within an
additive constant of $\frac{3n}{\delta+1}$.
For graphs with large minimum degree, we can show that {\it all} edges
of such graphs can be covered by a small number of low
diameter subgraphs (which are not necessarily edge-disjoint).

\begin{definition}
Let $Q(n,\epsilon,d)$ be the minimum $\ell$ such that the edges of any graph $G=(V,E)$ on $n$ vertices with minimum
degree at least $\epsilon n$ can be covered by $\ell$ sets $E=E_1 \cup \ldots \cup E_{\ell}$ such
that the diameter of each $E_i$ is at most $d$.
\end{definition}

We prove the following two results, in which we use the properties of Kneser graphs to establish lower bounds.

\begin{theorem}\label{thmq34}
(a) For fixed $0<\epsilon \leq 2^{-8}$, both $Q(n,\epsilon,3)$ and $Q(n,\epsilon,4)$ have order of magnitude
$\Theta(\log n)$. \\
(b) $Q(n,\epsilon,5)=\Theta(1/\epsilon^{2})$ and $Q(n,\epsilon,6)=\Theta(1/\epsilon)$.
\end{theorem}

Note that any graph on $n$ vertices of minimum degree more than $\frac{n}{2}-1$ has diameter at most $2$, since
non-adjacent vertices must have a common neighbor. In sharp contrast with Theorem
\ref{thmq34}(a), this shows that $Q(n,1/2,2)=1$.

An analogous problem for hypergraphs was first investigated by Polcyn, R\"odl, Ruci\'nski, and Szemer\'edi \cite{PRRS}. A hypergraph
$G=(V,E)$ is {\it $k$-uniform} if each edge has exactly $k$ vertices. A (tight) {\it path} of length $\ell$ in a $k$-uniform hypergraph
$G=(V,E)$ is a subhypergraph consisting of $\ell+k-1$ vertices $v_1,\ldots,v_{\ell+k-1}$  and $\ell$ edges, such that for each $i
\leq \ell$, $(v_i,v_{i+1},\ldots,v_{i+k-1})$ is an edge of $G$. The vertices $v_1$ and $v_{\ell+k-1}$ are the endpoints of the path.
The {\it distance} between two vertices $v,w$ in a hypergraph is the length of the shortest path whose endpoints are $v$ and $w$.
The {\it diameter} $\textrm{diam}(G)$ of $G$ is the maximum distance between any two vertices of $G$.

\begin{definition} Let $P_k(n,\epsilon,d)$ be the smallest $\ell$ such that every $k$-uniform hypergraph $G=(V,E)$ on $n$ vertices
has an edge partition $E=E_0 \cup E_1 \cup \ldots \cup E_k$ such that $|E_0| \leq \epsilon n^k$ and the diameter of $E_i$ is at most
$d$ for $1 \leq i \leq k$.
\end{definition}

Polcyn et al.~\cite{PRRS} showed that $P_3(n,\epsilon,12)$ is bounded above by a constant $C_3(\epsilon)$ depending only on
$\epsilon$. Their proof uses the hypergraph regularity lemma, and gives an Ackermann-type upper bound on $C_3(\epsilon)$. Here we improve the
diameter bound from $12$ to $3$, which is best possible,
generalize it to any uniformity $k$, and further show that this function is polynomial in $\epsilon^{-1}$.
The proof uses similar counting arguments as done in the graph case.

\begin{theorem}\label{mainhyper}
We have $P_k(n,\epsilon,3)=O\left(\epsilon^{2-2^k}\right)$.
\end{theorem}\
In the other direction, we show that $P_k(n,\epsilon,3) \geq c_k \epsilon^{-k}$ for $\epsilon \gg n^{-1/2}$, which we think is
tight.

We study $P(n,\epsilon,d)$ in the next section, where
we consider the cases $d=2,3,4$ in three separate subsections. In Section \ref{largemindegreesection}, we study
edge partitions of graphs of large minimum degree into a small number of low diameter subgraphs.
In Section \ref{hypergraphsection}, we prove bounds on $P_k(n,\epsilon,d)$.
The last section of the paper contains some concluding remarks and open questions.
Throughout the paper, we systematically omit floor and ceiling signs whenever they are not
crucial for the sake of clarity of presentation. We also do not make any serious attempt to optimize absolute constants in our
statements
and proofs. All logarithms in this paper are in base $2$.

\section{Proof of Theorems \ref{main1}}

In this section, we prove bounds on $P(n,\epsilon,d)$. We consider the cases $d=2,3,4$ in separate subsections.

\subsection{Decomposing a graph into diameter 2 subgraphs}

In this subsection we prove Theorem \ref{main1}(a), which states that if $\epsilon<1/4$ is bounded away from $1/4$,
then $P(n,\epsilon,2)=\Theta(n)$. The proofs for both the upper bound and for the lower bound are quite simple. We
first show the upper bound. Let $G$ be a graph with $n$ vertices. Note that a star has diameter $2$. Letting $E_i$ be
those edges that contain the $i$th vertex of $G$, we have $P(n,0,2) \leq n$.

To prove the lower bound we use the following simple fact.
If $G$ is a bipartite graph, then
any subgraph of $G$ which is not complete has diameter at least $3$. We show
next that if $G$ is a random bipartite graph of edge density bounded away from $1$, then almost surely every complete
bipartite subgraph of $G$ has $O(n)$ edges. The {\it random bipartite graph} $G(n,n,p)$ is the probability space of
labeled bipartite graphs with $n$ vertices in each class, where each of the $n^2$ edges appears independently with
probability $p$. The term {\it almost surely} means with probability tending to $1$ as $n$ tends to infinity.

\begin{lemma}
Fix $0<p<1$ and let $q=1-p$. Almost surely, all complete bipartite subgraphs of the random bipartite graph $G=G(n,n,p)$ have at most
$2n/q$ edges.
\end{lemma}
\begin{proof}
Let $A$ and $B$ be the two vertex classes of $G$. The probability that there is $S \subset A$ and $T \subset B$ that
are complete to each other and have at least $2n/q$ edges between them is at most $$p^{-2n/q}2^{2n} =
(1-q)^{-2n/q}2^{2n} \leq e^{-2n}2^{2n} = o(1).$$
\end{proof}

The Chernoff bound for the binomial distribution implies that the random graph $G(n,n,p)$ almost surely has
$pn^2+o(n^2)$ edges. Let $\epsilon<1/4$, $q=\frac{1}{4}-\epsilon$ and $p=\frac{3}{4}+\epsilon$. By considering
$G(n/2,n/2,p)$, we have that there is a bipartite graph $G$ on $n$ vertices with at least $(p-o(1))n^2/4
=\big(\epsilon+\frac{3q}{4}-o(1)\big)n^2$ edges and such that every diameter $2$ subgraph of $G$ has at most $2n/q$
edges. To
cover
$\big(\frac{3q}{4}-o(1)\big)n^2>qn^2/2$ edges of $G$ by diameter $2$ subgraphs, we need to use at least
$\frac{qn^2/2}{2n/q}=q^2n/4$ subgraphs. Hence
$P(n,\epsilon,2) \geq q^2n/4=\frac{(1-4\epsilon)^2}{64}n$ for $n$ sufficiently large. This completes the proof of
Theorem \ref{main1}(a). \qed

We conclude this subsection by showing that a sharp transition for $P(n,\epsilon,2)$ occurs at $\epsilon=1/4$. Namely,
$P(n,1/4,2)=1$. If a graph $G$ has a vertex which is adjacent to less than half of the other vertices, delete it.
Continue deleting vertices until all vertices in the remaining induced subgraph $G_1$ are adjacent to at least half of
the other vertices of $G_1$. Graph $G_1$ has diameter at most $2$ by the discussion after Theorem \ref{thmq34}. Let
$G_0$ be the subgraph whose edges are those containing a deleted vertex. Then $G_0$ has at most $\frac{1}{2}{n \choose
2} < n^2/4=\epsilon n^2$ edges and therefore $P(n,1/4,2)=1$.

\subsection{Decomposing a graph into diameter 3 subgraphs}
\label{sectdiam3}

Studying $P(n,\epsilon,d)$ appears to be most interesting in the case $d=3$. In this subsection, we prove Theorem \ref{main1}(b),
which establishes
$P(n,\epsilon,3)=\Theta(1/\epsilon^2)$ for $\epsilon \geq n^{-1/2}$. We begin with a few simple lemmas.

\begin{lemma} \label{lem123}
If $G$ is a graph with $n$ vertices and $m \geq 12n$ edges, then $G$ has at least $\frac{m^3}{16n^2}$ paths of length three.
\end{lemma}
\begin{proof}
Delete from $G$ vertices of degree at most $\frac{m}{2n}$ one by one. The resulting induced subgraph $G'$
has at least $m-n\frac{m}{2n}=m/2$
edges, and has minimum degree at least $\frac{m}{2n}$. For each edge $e=(u,v)$ of $G'$, $u$ and $v$ each have at least
$\frac{m}{2n}$ neighbors, so the number of paths
of length three with middle edge $e$ is at least $(\frac{m}{2n}-1)(\frac{m}{2n}-2) \geq \frac{m^2}{8n^2}$ as there are at least
$\frac{m}{2n}-1$ possible choices for the first vertex of the path, and, given the first three vertices of the path, at least
$\frac{m}{2n}-2$ remaining possible last vertices for the path. Counting over all $m/2$ possible
middle edges, we obtain that the number of paths of length three in $G'$ (and hence in $G$) is at least
$\frac{m}{2}\frac{m^2}{8n^2}=\frac{m^3}{16n^2}$.
\end{proof}

The next definition demonstrates how to construct a subgraph of diameter at most $3$ from a graph and a pair of its vertices.

\begin{definition}\label{defdiam}
For a graph $G$ and vertices $v$ and $w$ of distance at most $d$, let $G_d(v,w)$ be the induced subgraph of $G$ consisting of all
vertices
of $G$ that lie on a walk from $v$ to $w$ of length at most $d$.
\end{definition}

\begin{lemma}\label{diamatmostd}
The graph $G_d(v,w)$ has diameter at most $d$.
\end{lemma}
\begin{proof}
Let $a$ and $b$ be vertices of $G_d(v,w)$. So $a$ is on a walk from $v$ to $w$ of length at most $d$, and $b$ is on a walk from $v$
to $w$ of length at most $d$. These two walks give rise to two walks from $a$ to $b$ such that the sum of the lengths of these two
walks is at most $2d$. Hence, there is a path from $a$ to $b$ of length at most $d$. This shows that $G_d(v,w)$ has diameter at most
$d$.
\end{proof}

\begin{lemma}\label{mainfor3}
If a graph $G$ has $n$ vertices and $m \geq 12n$ edges, then it contains an induced subgraph $H$ with at least $\frac{m^3}{32n^4}$
edges that has diameter at most $3$.
\end{lemma}
\begin{proof}
By Lemma \ref{lem123}, $G$ has at least $\frac{m^3}{16n^2}$ paths of length three. By
averaging, there is a pair $u,v$ of vertices of $G$ such that the number of paths of length three
with terminal vertices $u$ and $v$ is at least
$\frac{m^3}{16n^4}$. By Lemma \ref{diamatmostd}, $G_3(u,v)$ has diameter at most $3$ and in every paths of length three
from $u$ to
$v$, the middle edge is an edge of $G_3(u,v)$. Moreover, each edge in $G_3(u,v)$ is the middle edge of at most two paths from $u$ to
$v$,
hence $G_3(u,v)$ has at least $\frac{m^2}{32n^4}$ edges.
\end{proof}

We now prove a quantitative version of the upper bound on $P(n,\epsilon,3)$ in Theorem \ref{main1}(b).

\begin{theorem}
\label{thmfor3}
Every graph $G$ on $n$ vertices can be edge partitioned $E=E_0 \cup E_1 \cup \ldots \cup E_k$ such that $|E_0| \leq \epsilon n^2$,
$k \leq 50\epsilon^{-2}$, and for $1 \leq i \leq k$, the diameter of $E_i$ is at most $3$.
\end{theorem}

\begin{proof} We repeatedly use Lemma \ref{mainfor3} to pull out subgraphs of diameter at most $3$ until the remaining subgraph has
at most $\epsilon n^2$ edges. The remaining at most $\epsilon n^2$ edges make up $E_0$.
If the current graph has at least $m/2$ edges, then by the above lemma we can find a subgraph of diameter at most $3$ with at
least $\frac{(m/2)^3}{32n^4}$ edges. Therefore, after pulling out
$s= (m-m/2)/\frac{(m/2)^3}{32n^4}=128n^4/m^2$ such subgraphs of diameter at most $3$, we remain with at most $m/2$ edges.
Similarly, applying this process to a subgraph with at most $2^i \epsilon n^2$ edges we get a subgraph with at most
$2^{i-1}\epsilon n^2$ edges, after pulling out at most
$\frac{128n^4}{(2^i\epsilon n)^2}=2^{7-2i}\epsilon^{-2}$ subgraphs of diameter $3$.
Summing over all $i \geq 1$, we obtain that altogether
we pull out at most $\sum_{i=1}^{\infty}
2^{7-2i}\epsilon^{-2}=\frac{2^7}{3}\epsilon^{-2}<50\epsilon^{-2}$ subgraphs.
\end{proof}

Next we establish a lower bound on $P(n,\epsilon,3)$, using the following two lemmas.

\begin{lemma}\label{Gnnp}
Almost surely the subgraph of diameter at most $3$ of the random bipartite graph $G(n,n,p)$ with $p=\frac{1}{4\sqrt{n}}$ with the
maximum number of
 edges has $(1+o(1))(2pn+p^3n^2)=\big(\frac{33}{64}+o(1)\big)\sqrt{n}$ edges.
\end{lemma}

The {\it neighborhood} $N(v)$ of a vertex $v$ in a graph $G$ is the set of vertices adjacent to $v$. To see that there is almost
surely a subgraph of diameter $3$ of $G(n,n,p)$ with $(1+o(1))(2pn+p^3n^2)$ edges, let $a$ and $b$ be adjacent vertices in different
vertex classes of $G(n,n,p)$, and consider the induced subgraph $H$ with vertex set $N(a) \cup N(b)$, which has diameter at most
$3$. It follows from Chernoff's bound for the binomial distribution that almost surely all vertices of $G(n,n,p)$ have degree
$(1+o(1))pn$. Hence, the number of edges containing $a$ or $b$ is $(1+o(1))2pn$. Furthermore,  given $|N(a)|$ and $|N(b)|$, the
number of edges between $N(a) \setminus \{b\}$ and $N(b) \setminus \{a\}$ follows a binomial distribution, and another application
of Chernoff's bound for the binomial distribution implies that the number of these edges is $(1+o(1))p^3n^2$.
Hence $H$ has $(1+o(1))(2pn+p^3n^2)$ edges. To prove Lemma \ref{Gnnp}, it thus suffices to show that {\it every} diameter
$3$ subgraph has at most $(1+o(1))(2pn+p^3n^2)$ edges.

Since almost surely the number of edges of $G(n,n,p),~p=\frac{1}{4\sqrt{n}}$ is concentrated around it
expected value $pn^2$, we have the following corollary. For all sufficiently large $n$, there is a bipartite graph on $2n$ vertices
with at least
$\frac{1}{5}n^{3/2}$ edges in which any diameter at most $3$ subgraph has at most $\frac{4}{5}n^{1/2}$ edges.

For a graph $G$, the {\it blow-up} $G(r)$ denotes the graph formed by replacing each vertex $v_i$ of $G$
by an independent set $V_i$ of size $r$, where vertices $u \in V_i$ and $w \in V_j$ are adjacent in $G(r)$
if and only if $v_i$ and $v_j$ are adjacent in $G$.

\begin{lemma}\label{blowuplem}
If any subgraph of a graph $G$ with diameter at most $d$ has at most $m$ edges,
then any subgraph of $G(r)$ with diameter at most $d$ has at most $r^2m$ edges.
\end{lemma}

\begin{proof} Let $H$ be a subgraph of $G(r)$ of diameter at most $d$. Let $H'$ be the induced subgraph of $G$ where $v_i$ is a
vertex of $H'$ if there is a vertex of $H$ in $V_i$. It is clear from the definition of $G(r)$ that the diameter of $H'$ is
at most the diameter of $H$. Thus $H'$ also has diameter at most $d$ and so it has at most $m$ edges.
Then $H$ has at most $r^2m$ edges,
since for every edge
$(v_i,v_j)$ of $H'$ there are at most $r^2$ edges of $H$ in $V_i
\times V_j$. \end{proof}

>From Lemmas \ref{Gnnp} and \ref{blowuplem}, we quickly deduce a lower bound on $P(n,\epsilon,3)$.
Of course, we are assuming here that $\epsilon<1/2$ as otherwise $\epsilon n^2 \geq {n \choose 2}$ and $P(n,\epsilon,3)=0$ since
we can let $E_0$ consist of all edges of the graph.

\begin{theorem}\label{diam3lwrb}
We have $P(n,\epsilon,3) \geq c\epsilon^{-2}$ for some absolute constant $c>0$.
\end{theorem}

\begin{proof}
Let $t=(40\epsilon)^{-2}$ and $r=\frac{n}{2t}$.
By choosing an appropriate constant $c$ we may suppose that $\epsilon$ is sufficiently small
and so $t$ is sufficiently large. Therefore, by Lemma \ref{Gnnp},
there is a bipartite graph $G$ with $2t$ vertices, at least $\frac{1}{5}t^{3/2}$ edges
such that every diameter $3$ subgraph of $G$ has
at most $\frac{4}{5}t^{1/2}$ edges. The blow-up
graph $G(r)$ has $n$ vertices, at least $\frac{1}{5}t^{3/2}r^2 = 2\epsilon n^2$ edges,
and Lemma \ref{blowuplem} shows that any subgraph of $G(r)$ with diameter at most $3$ has at most $\frac{4}{5}r^2t^{1/2} =
t^{-3/2}n^2/5=8(40)^2\epsilon^3n^2$ edges.
Thus $P(n,\epsilon,3) \geq \frac{\epsilon n^2}{8(40)^2\epsilon^3 n^2} \geq \frac{1}{8(40)^2}\epsilon^{-2}$, which completes the
proof.
\end{proof}

We next include a simple characterization of bipartite graphs of diameter at most $3$.

\begin{proposition}\label{diam3prop}
A bipartite graph $G$ with at least three vertices has diameter at most $3$ if and only if each pair of vertices in the same vertex
class have a common neighbor.
\end{proposition}

\begin{proof} Let $A$ and $B$ be the vertex classes of bipartite $G$ with $|A| \geq |B|$. If a pair of vertices in the same vertex
class does not have a common neighbor, then the shortest path between them must be even, have length at least four, and hence $G$
has diameter at least four.

Conversely, suppose each pair of vertices in the same vertex class have a common neighbor. If two vertices are in the same vertex
class, then there is a path of length two between them. If $a \in A$ and $b \in B$, then there is another vertex $a' \in A$, and
hence $a$ and $a'$ have a neighbor $b'$ in $B$. If $b'=b$, then there is a path of length one between $a$ and $b$. Otherwise, $b'$
and $b$ have a common neighbor, so there is a path between $a$ and $b$ of length at most three. This shows that the distance between
any pair of vertices of $G$ is at most $3$, i.e., $G$ has diameter at most $3$. \end{proof}

Notice that a bipartite graph with diameter less than $3$ must be a complete bipartite
graph, as if there is a pair of vertices in different classes that are not adjacent, then the shortest path between them is of
odd length greater than $1$, and so the diameter is at least $3$. Hence a bipartite graph with at least three vertices has diameter
exactly three if and only if
it is not a complete bipartite graph and each pair of vertices in the same vertex class have a common neighbor.

Our goal for the rest of the subsection is to prove Lemma \ref{Gnnp}. We will assume that $p=\frac{1}{4\sqrt{n}}$, $n$ is
sufficiently large,
and let $A$ and $B$ denote the vertex sets of size $n$ of $G(n,n,p)$. We first need to collect several basic lemmas about the edge
distribution in $G(n,n,p)$.

\begin{lemma}\label{lem1} $G(n,n,p)$ almost surely has the following six properties. \\
(a) Every vertex has degree $(1+o(1))pn$. \\
(b) Every pair of vertices have at most $\log n$ common neighbors.\\
(c) For every edge $(a,b)$, there are $(1+o(1))(2pn+p^3n^2)$ edges between $N(a)$ and $N(b)$.\\
(d) For all $a \in A$ and $b \in B$ non-adjacent, there are $(1+o(1))p^3n^2$ edges between $N(a)$ and
$N(b)$.\\
(e) For all $Y \subset B$ and $X \subset A$, there are at most $|X|+|Y|^2 \log n$ edges between $X$ and $Y$.\\
(f) For all $A' \subset A$ and $B' \subset B$, there are at most
$t=6\max(|A'||B'|p,(|A'|+|B'|)\log n)$ edges between $A'$ and $B'$.
\end{lemma}
\begin{proof}
\noindent (a) The degree of each of the $2n$ vertices of $G(n,n,p)$ follows a binomial distribution. Chernoff's bound for the
binomial distribution implies that almost surely all of the degrees are concentrated around their expected value, $pn$.
\vspace{0.05cm}

\noindent (b) The probability that a given pair of vertices in the same part have at least $\log n$ common neighbors is at most
${n \choose \log n} p^{2\log n}
<n^{\log n}(4\sqrt{n})^{-2\log n}=n^{-4}$. There are $2{n \choose 2} <n^2$ pairs of vertices in the same class of $G(n,n,p)$, so
the probability $G(n,n,p)$ has a pair of vertices with $\log n$ common neighbors is at most $n^{-4}n^2 = n^{-2}$.
\vspace{0.05cm}

\noindent (c) By (a), almost surely, the degree of every vertex in $G(n,n,p)$ is $(1+o(1))pn$. If $a$ and $b$ are adjacent,
given $|N(a)|$ and $|N(b)|$, the number of edges between $N(a) \setminus \{b\}$ and $N(b) \setminus \{a\}$ follows a binomial
distribution. An application of Chernoff's bound for the binomial distribution implies that a.s.~for each edge $(a,b)$ the number of
edges between $N(a)$ and $N(b)$ is $|N(a)|+|N(b)|-1+(1+o(1))p|N(a)||N(b)|=(1+o(1))(2pn+p^3n^2)$.
 \vspace{0.05cm}

\noindent (d) Similar to (c), an application of Chernoff's bound for the binomial distribution implies that almost surely for each
pair $a,b$ of non-adjacent vertices in different vertex classes, the number of edges between $N(a)$ and $N(b)$ is
$(1+o(1))p|N(a)||N(b)|=(1+o(1))p^3n^2$.  \vspace{0.05cm}

\noindent (e) Let $x_1, \ldots, x_k$ be the vertices of $X$ with at least two neighbors in $Y$.
Since every $d_Y(x_i)\geq 2$ and (by (b)) each pair of vertices in $Y$ has at most $\log n$ common neighbors, we have that
$$ \frac{1}{2}\sum_i d_Y(x_i) \leq \sum_i {d_Y(x_i) \choose 2} \leq {|Y| \choose 2}\log n.$$
This implies that the total number of edges between $X$ and $Y$ is at most
$|X|-k+\sum_i d_Y(x_i) \leq |X|+|Y|^2 \log n$.
\vspace{0.05cm}

\noindent (f) The result is trivial if $A'$ or $B'$ is empty, so we may assume they are nonempty. The number of pairs $A' \subset A$
and $B' \subset B$ of size $|A'|=a$ and $|B'|=b$ is ${n \choose a}{n \choose b}$.
For fixed $A'$ and $B'$ of sizes $a$ and $b$, respectively,
the probability that there are at least $t$ edges between them is at most $p^t{ab \choose t} \leq \left(\frac{abep}{t}\right)^t \leq
2^{-t} \leq n^{-6(a+b)}$. So the probability that there are such subsets $A'$ and $B'$
is at most $\sum_{a,b}{n \choose a}{n \choose b}n^{-6(a+b)} \leq \sum_{a,b} n^{a+b}n^{-6(a+b)} =\sum_{a,b} n^{-5(a+b)}
\leq n^2n^{-5}=n^{-3}$.
\end{proof}

Let $H$ be a subgraph of the bipartite graph $G=G(n,n,p)$ of diameter at most $3$, and $X$ and $Y$ be its vertex sets with $|X| \geq
|Y|$.
In particular, according to Proposition \ref{diam3prop}, each pair of vertices of $X$ have a common neighbor in $Y$,
and each pair of vertices of $Y$ have a common neighbor in $X$. We suppose for contradiction that $H$ has more than
$(1+o(1))(2pn+p^3n^2)$ edges.
Adding extra edges to a graph on a given vertex set cannot increase the diameter of the graph, so we may suppose that $H$ is the
induced subgraph of $G$ with vertex sets $X$ and
$Y$. In graph $H$, let $d_1 \geq d_2 \geq \ldots \geq d_{|Y|}$ be the degrees of the vertices of $Y$ in decreasing order, and $v_i$
denote the vertex of degree $d_i$.

We first prove that $H$ has few vertices.

\begin{claim}\label{claim1}
Almost surely, every diameter $3$ subgraph $H$ of $G(n,n,p)$ has at most $10n^{1/2}\log n$ vertices.
\end{claim}

\begin{proof} Suppose for contradiction that $H$ has at least $x=10n^{1/2}\log n$ vertices. Since $|X| \geq |Y|$, we have $|X| \geq
x/2 =5n^{1/2}\log n$. Lemma \ref{lem1}(f) shows that almost surely there are at most $t=6\max(|X||Y|p,(|X|+|Y|)\log n) \leq
6|X|\max\left(p|X|,2\log n\right)$ edges between $X$ and $Y$. Also, by Lemma \ref{lem1}(a), a.s.~the maximum degree satisfies
$\Delta = (1+o(1))np$.
Convexity of the function $f(y)={y \choose 2}$ yields $$\sum_{i=1}^{|Y|} {d_i \choose 2} \leq \frac{t}{\Delta}{\Delta \choose 2} =
(1+o(1))npt/2 \leq 4np|X|\max\left(p|X|,2\log n\right),$$ which is an upper bound on the number of pairs of vertices of $X$ that
have
a common neighbor in $Y$. We have $4np|X| \cdot p|X| <{|X| \choose 2}$ and $4np|X|\cdot 2\log n <{|X| \choose 2}$. Hence,
there are less than ${|X| \choose 2}$ pairs of vertices in $X$ with a common neighbor in $Y$. So there is a pair of
vertices in $X$ with no common neighbor in $Y$, contradicting $H$ has diameter at most $3$ and completing the proof. \end{proof}

Lemma \ref{lem1}(f) together with the previous claim imply that a.s.~any subgraph of $G(n,n,p)$ of diameter $3$ with at least
$2pn=\sqrt{n}/2$ edges has at least $\frac{1}{12}\frac{\sqrt{n}}{\log n}$ vertices and at most $10\sqrt{n}\log n$ vertices.

\begin{claim}\label{claim2}
Let $\epsilon=|X|^{-1/5}$. There is a vertex in $Y$ that has at least $(1-\epsilon)|X|$ neighbors in $X$.
\end{claim}
\begin{proof}
Since every pair of vertices of $X$ have a common neighbor in $Y$ we have $\sum_{i=1}^{|Y|} {d_i \choose 2} \geq {|X| \choose 2}$,
where $d_1 \geq d_2\geq \ldots \geq d_{|Y|}$ are the degrees of vertices
of $Y$ in $X$. Suppose for contradiction that no vertex of $Y$ that has at least $(1-\epsilon)|X|$ neighbors in $X$.
Let $r=|X|^{1/3}$. By Lemma \ref{lem1}(e),
$\sum_{i=1}^r d_i \leq |X|+r^2\log n$. In particular, convexity of the function $f(x)={x \choose 2}$ demonstrates that
$$\sum_{i=1}^r {d_i \choose 2}  \leq  {(1-\epsilon)|X| \choose 2}+{\epsilon|X|+r^2 \log n \choose 2} \leq
{(1-\epsilon)|X| \choose 2}+\epsilon^2|X|^2  \leq  {|X| \choose 2}-\epsilon |X|^2/2.$$

Since $H$ and hence also $X$ has at most $10\sqrt{n}\log n$ vertices, we have that $|X|^2p \leq \frac{5}{2}|X|\log n$.
Thus, by Lemma \ref{lem1}(f), there are at most $15 |X|\log n$ edges from $Y$ to $X$.
This implies that $d_i \leq d_r \leq 15 |X|\log n/r$ for all $i >r$.
Under these constraints,  we have by convexity of the function $f(x)={x \choose 2}$ that
$$\sum_{i>r} {d_i \choose 2} \leq r {15 |X|\log n/r \choose 2} <
\frac{120}{r} |X|^{2}\log^2 n=120 |X|^{5/3}\log^2 n \ll \frac{1}{2}|X|^{9/5}=\frac{1}{2}\epsilon |X|^2.$$
This together with the above estimate shows that $\sum_{i=1}^{|Y|} {d_i \choose 2} < {|X| \choose 2}$, a contradiction.
\end{proof}

Take $\epsilon$ as in Claim \ref{claim2}. Since there is a vertex in $Y$ with at least
$(1-\epsilon)|X|$ neighbors in $X$ and every vertex has degree $(1+o(1))pn$, then $|X| \leq (1+o(1))pn/(1-\epsilon) = (1+o(1))pn$.

We next show that $Y$ is also quite large if $H$ has at least $2pn$ edges. By Claim \ref{claim2}, there is a vertex $v \in Y$
adjacent to at least $(1-\epsilon)|X|$ elements of $X$. By Lemma \ref{lem1}(b), every other vertex in $Y$ besides $v$ has at most
$\log n$ neighbors in $N(v)$. Hence, there are at most $|X|+|Y|\log n$
edges between $N(v)$ and $Y$. There are at most $\epsilon|X|$ vertices in $X \setminus N(v)$, so there are at most
$6\max(\epsilon|X||Y|p,(\epsilon|X|+|Y|)\log n)$
edges between $X \setminus N(v)$ and $Y$. Hence the number of edges of $H$ is at most
$$|X|+|Y|\log n+6\max(\epsilon|X||Y|p,(\epsilon|X|+|Y|)\log n) \leq (1+o(1))pn+7|Y|\log n,$$
where we use $|X|\leq (1+o(1))pn$. If $|Y| \leq \frac{1}{30}\frac{\sqrt{n}}{\log n}$, we get that there are less than $2pn$
edges in $H$, a contradiction. To summarize, we have the following inequalities:
$\frac{\sqrt{n}}{30\log n} \leq |Y| \leq |X| \leq (1+o(1))pn$.

Now that we have established that $X$ and $Y$ are of similar size, the proof of Claim \ref{claim2} with $X$ and $Y$ switched also
gives us the following claim. The only estimate in the proof that needs to be checked is that $|X|^2(\log n)^2/r \ll \epsilon
|Y|^2$, and since $|X|$ and $|Y|$ are both of the form $n^{1/2+o(1)}$, $r=|Y|^{1/3}$ and $\epsilon=|Y|^{-1/5}$, this clearly holds.

\begin{claim}\label{claim3}
Let $\epsilon=|Y|^{-1/5}$. There is a vertex in $X$ that has at least $(1-\epsilon)|Y|$ neighbors in $Y$.
\end{claim}

We now complete the proof of Lemma \ref{Gnnp}. Let $y$ be a vertex in $Y$ with at least $(1-\epsilon)|X|$ neighbors in $X$ and $x$
be a
vertex in $X$ with at least $(1-\epsilon)|Y|$ neighbors in $Y$, where $\epsilon=|Y|^{-1/5}$. Such vertices $x$ and $y$ exist
by Claims
\ref{claim2} and \ref{claim3}, since $|Y| \leq |X|$. Let $X_1$ be the set of neighbors of $y$ in $X$, and $X_2=X \setminus X_1$. Let
$Y_1$ denote the set of neighbors of $x$ in $Y$, and $Y_2=Y \setminus Y_1$. By Lemma \ref{lem1}(c) if $x$ and $y$ are adjacent and Lemma
\ref{lem1}(a) and (d) if $x$ and $y$ are not adjacent,  there are at most $(1+o(1))(2pn+p^3n^2)$ edges between
$X_1 \cup \{x\}$ and $Y_1 \cup \{y\}$.  Since $X_1$ consists of neighbors of $y$, by Lemma \ref{lem1}(b), each vertex in $Y_2
\setminus \{y\}$ has at most $\log n$ neighbors in $X_1$. Similarly, each vertex in $X_2 \setminus \{x\}$ has at most $\log n$
neighbors in $Y_1$. Lemma \ref{lem1}(f) implies that the number of edges between $X_2$ and $Y_2$ is at most
$$6\max(|X_2||Y_2|p,(|X_2|+|Y_2|)\log n) \leq 6\max(\epsilon|X|\epsilon|Y|p,(\epsilon|X|+\epsilon|Y|)\log n) < 20\epsilon|X|\log n
=o(n^{1/2}).$$
Putting these inequalities altogether, the number of edges between $X$ and $Y$, and hence the number of edges of $H$, is at most
$$(1+o(1))(2pn+p^3n^2)+|X_2|\log n+|Y_2|\log n + o(n^{1/2}) =(1+o(1))(2pn+p^3n^2),$$
where we use $|X_2| \leq \epsilon|X|=o(n^{1/2})$ and $|Y_2| \leq \epsilon|Y|=o(n^{1/2})$, which completes the proof. \qed

\subsection{Decomposing a graph into diameter 4 subgraphs}

In this subsection, we prove the last claim of Theorem \ref{main1}, that $P(n,\epsilon,4)=\Theta(1/\epsilon)$ for
$\epsilon \geq 1/n$.

\label{sectdiam4}
\begin{definition}
For a vertex $v$ and graph $G$, let $N_r(v)$ be those vertices of $G$ which are within distance at most $r$ from vertex $v$ and
$G_r(v)$ denote the induced subgraph of $G$ with vertex set $N_r(v)$.
\end{definition}

The graph $G_r(v)$ of course has radius at most $r$ and hence diameter at most $2r$. Note that $G_r(v)=G_{2r}(v,v)$ (defined in the
previous section) as any vertex at distance at most $r$ from $v$ is contained in a walk from $v$ to $v$ of length at most $2r$, and
any walk from $v$ to $v$ of length at most $2r$ contains only vertices of distance at most $r$ from $v$.

To bound $P(n,\epsilon,4)$ from above we use the following lemma, which
 is tight apart from the constant factor as demonstrated by a disjoint union of cliques of equal
size (see the proof of Lemma \ref{last} below).

\begin{lemma}\label{secondlemmafor4}
If a graph $G$ has $n$ vertices and at least $m \geq 4n$ edges, then it has a subgraph with diameter at most $4$ and at least
$\frac{m^2}{8n^2}$ edges.
\end{lemma}
\begin{proof}
Delete vertices one by one of degree at most $\frac{m}{2n}$. The resulting induced subgraph $G'$ has at least $m-n\frac{m}{2n}=m/2$
edges and minimum degree at least $\frac{m}{2n}$. For any vertex $v$ of $G'$, $G'_2(v)$ has diameter at most $4$ and at least
$(\frac{m}{2n})^2/2 \geq \frac{m^2}{8n^2}$ edges as $v$ and its neighbors have degree at least $\frac{m}{2n}$.
\end{proof}

We now prove a quantitative version of the upper bound on $P(n,\epsilon,4)$ in Theorem \ref{main1}(c).

\begin{theorem}\label{thmfor4}
Every graph $G$ on $n$ vertices can be edge partitioned $E=E_0 \cup E_1 \cup \ldots \cup E_{\ell}$ such that $|E_0| \leq \epsilon
n^2$, $\ell \leq 16\epsilon^{-1}$, and for $1 \leq i \leq \ell$, the diameter of $E_i$ is at most $4$.
\end{theorem}

\begin{proof} We repeatedly use Lemma \ref{secondlemmafor4} to pull out subgraphs of diameter at most $4$ until the remaining
subgraph has at most $\epsilon n^2$ edges. The remaining at most $\epsilon n^2$ edges make up $E_0$. If the current graph has at
least $m/2$ edges, then by the above lemma we can pull out a subgraph of diameter at most $4$ with at least
$(m/2)^2/8n^2$ edges. Therefore, after pulling out
$s= (m-m/2)/\frac{(m/2)^2}{8n^2}=16n^2/m$ subgraphs of diameter at
most $4$ from our graph, we remain with at most $m/2$ edges.
Similarly, applying this process to a graph with at most $2^i \epsilon n^2$ edges we get a subgraph
with at most $2^{i-1}\epsilon n^2$ edges, after we pull out
at most $\frac{16n^2}{2^i\epsilon n^2}=2^{4-i}\epsilon^{-1}$
subgraphs of diameter $4$.
Summing over all $i \geq 1$, the total number $\ell$ of subgraphs of diameter at most $4$ we pull out is at
most $\sum_{i=1}^{\infty} 2^{4-i}\epsilon^{-1}=16\epsilon^{-1}.$ \end{proof}

The next lemma shows that Theorem \ref{thmfor4} is tight apart from a constant factor.

\begin{lemma} \label{last}
If $\frac{1}{4n}\leq \epsilon \leq \frac{1}{16}$, then $P(n,\epsilon,d) \geq \frac{1}{16\epsilon}$.
\end{lemma}

\begin{proof} Let $t=\frac{1}{8\epsilon}$ (note that $t \leq n/2$) and let $G$ be a graph consisting of $n$ vertices partitioned
into
$t$
disjoint cliques each
of size $n/t$. The graph $G$ has $t{n/t \choose 2}$ edges. Any connected subgraph of $G$ has
${n/t \choose 2}$ edges. Hence, any partition $E=E_0 \cup E_1 \cup \ldots \cup E_{\ell}$ such that $|E_0| \leq \epsilon n^2
 \leq \frac{t}{2}{n/t \choose 2}$ and each $E_i$ is
connected for $1 \leq i \leq \ell$ must have $\ell \geq t/2 = \frac{1}{16\epsilon}$. \end{proof}

\section{Covering graphs of large minimum degree}\label{largemindegreesection}

In this section, we prove results on the minimum number of subgraphs needed to cover {\em all the edges} of a graph of large minimum
degree by low diameter
subgraphs. First we show, using a simple sampling argument, that the edges of every graph with $n$ vertices and minimum degree
linear in $n$ can
be covered by $O(\log n)$ subgraphs of diameter at most $3$.

\begin{theorem}\label{lemsampling}
Let $G=(V,E)$ be a graph on $n$ vertices with minimum degree at least $\epsilon n$. Then there is a covering $E=E_1 \cup \ldots \cup
E_{\ell}$ of the edge set of $G$ such that $\ell = 2\epsilon^{-2}\log n$ and for $1 \leq i \leq \ell$, $E_i$ has diameter at
most
$3$.

\end{theorem}
\begin{proof}
Pick $\ell$ pairs $(v_i,w_i)$ of not necessarily distinct vertices uniformly at random with repetition.
If there is a path of length at most $3$ between $v_i$ and $w_i$,
let $G_3(v_i,w_i)$ be the induced subgraph of $G$ as in Definition \ref{defdiam} and $E_i$ be the edge set of $G_3(v_i,w_i)$.
By Lemma \ref{diamatmostd}, $G_3(v_i,w_i)$ has diameter at most $3$.
Let $e=(v,w)$ be an edge of $G$.
If $v_i \in \{v\} \cup N(v)$ and $w_i \in \{w\} \cup N(w)$, then $G_3(v_i,w_i)$ contains $e$ as there is a walk from $v_i$ to $w_i$
of length at most
$3$ containing $e$. So the probability that a given edge $e=(v,w)$ of $G$ is in $G_3(v_i,w_i)$ is at least
$\frac{|N(v)|}{n}\frac{|N(w)|}{n} \geq \epsilon^{2}$.
Since we are picking $\ell=2\epsilon^{-2}\log n$ pairs of vertices uniformly at random,
the probability $e$ is in none of these subgraphs is at most $(1-\epsilon^2)^{\ell}<e^{-2\log n}=n^{-2}$.
Summing over all edges of $G$, the expected number of edges of $G$ that are not in any of the $E_i$ is at most ${n \choose
2}n^{-2}<1/2$.
Hence, there is a choice of $E_1,\ldots,E_{\ell}$, each having diameter at most $3$, that together cover all edges of $G$. This
completes the proof.
\end{proof}

The family of all subsets of $[n]=\{1,\ldots,n\}$ of size $k$ is denoted by ${[n] \choose k}$. The {\it Kneser graph} KG$(n,k)$ has
vertex set ${[n] \choose k}$, where two sets of size $k$ are adjacent if they are disjoint. A famous theorem of
Lov\'asz \cite{Lo}, who proved Kneser's conjecture using topological methods,
states that the chromatic number of KG$(n,k)$ is $n-2k+2$ for $n \geq 2k \geq
2$. We use this property of Kneser graphs to construct a graph demonstrating that the bound in Theorem \ref{lemsampling} is tight up
to a constant factor. This will complete the proof of Theorem \ref{thmq34}(a).

\begin{theorem}\label{largemindegreediam4}
For every sufficiently large $N$, there is a graph $G$ on $N$ vertices with minimum degree at least $2^{-8}N$ whose edges cannot be
covered by $\frac{1}{2}\log_2 N$ subgraphs of diameter at most $4$.
\end{theorem}

We establish Theorem \ref{largemindegreediam4} using the following lemma.

\begin{lemma} \label{lemmalastfordiam4}
Let $a_k=4{4k \choose k}$. Then there is a graph $F_k$ on $a_k$ vertices with minimum degree
$a_k/16$ such that any covering of the edges of $F_k$ by subgraphs of diameter at most $4$ uses at least $2k+2$ subgraphs.
\end{lemma}

To deduce Theorem \ref{largemindegreediam4} from Lemma \ref{lemmalastfordiam4}, we let $k$ be the largest positive
integer such that $a_k \leq N$. If $a_k$ is not exactly $N$, we can duplicate some vertices of $F_k$ if necessary to get the desired
graph $G$ with $N$ vertices.
Indeed, as $N<a_{k+1} \leq 10a_{k}$ and $a_k = 4{4k \choose k} \leq 2^{4k}$, we have that $G$ has minimum degree at least
$a_k/16 \geq 2^{-8}N$ and the edges of $G$
cannot be covered by less than $2k+2 > \frac{1}{2}\log_2 N$ subgraphs of diameter at most $4$. It thus suffices to prove Lemma
\ref{lemmalastfordiam4}.

The {\it incidence graph} IG$(n,k)$ is a bipartite graph with first vertex class ${[n] \choose k}$ and second vertex class $[n]$,
where $i \in [n]$ is adjacent to $S \subset {[n] \choose k}$ if $i \in S$. Every vertex in the first vertex class has $k$ neighbors,
and every vertex in the second vertex
class has ${n-1 \choose k-1}=\frac{k}{n}{n \choose k}$ neighbors. Two vertices $S_1,S_2 \in {[n] \choose k}$ have a common neighbor
in IG$(n,k)$ if and only if they have nonempty intersection. So any coloring of ${[n] \choose k}$ such that any pair $S_1,S_2$  of
the same color have a neighbor in common in IG$(n,k)$ gives a proper vertex coloring of KG$(n,k)$. Since IG$(n,k)$ is bipartite, the
number of colors of any coloring of ${[n] \choose k}$ in which any pair $S_1,S_2$ of the same color have distance less than $4$ in
IG$(n,k)$ is at least the chromatic number of KG$(n,k)$, which is $n-2k+2$.

Let H$(n,k,t)$ denote the graph with first vertex class ${[n] \choose k}$ and second vertex class $[nt]$, where the bipartite graph
between ${[n] \choose k}$ and $\{n(j-1)+1,\ldots,nj\}$ make a copy of $IG(n,k)$ for $1 \leq j \leq t$. By construction, any coloring
of ${[n] \choose k}$ in which any pair $S_1,S_2$ of the same color have distance less than $4$ in H$(n,k,t)$ uses at least $n-2k+2$
colors. Let $t=\frac{1}{4k}{4k \choose k}$ and $H_k=H(4k,k,t)$. The number of vertices of $H_k$ is $a_k/2$, and every vertex
has degree $\frac{1}{4}{4k \choose k}=a_k/16$. We have established the following lemma.

\begin{lemma}\label{lastlemmamindegreediam4}
The bipartite graph $H_k$ is $a_k/16$-regular with $a_k/2$ vertices, and every coloring of the first vertex class of
$H_k$ such that every pair of vertices of the same color have distance less than $4$ uses at least $2k+2$ colors.
\end{lemma}

Let $F_k$ be the graph consisting of two disjoint copies $H_k^1$, $H_k^2$ of $H_k$ with an edge between the two copies $S^1$ and
$S^2$ of $S$ for each vertex $S \in {[n] \choose k}$ in the first vertex class of $H_k$. Notice that $F_k$ has $2|H_k| =a_k$
vertices and has minimum degree $a_k/16$. To complete the proof of Lemma \ref{lemmalastfordiam4} and hence of Theorem
\ref{largemindegreediam4}, it suffices to show that in any covering $E(F_k)=E_1 \cup \ldots \cup E_{\ell}$ in which each $E_i$ has
diameter at most $4$, the number $\ell$ of subgraphs used is at least $2k+2$.
Given such an edge-covering of $F_k$, define a coloring $\chi:{[n] \choose k} \longrightarrow [\ell]$ as follows. Let $\chi(S)=i$ if
$i$ is the smallest positive integer such that the edge between the two copies of $S$ in $F_k$ is in $E_i$.

The key observation is that if $\chi(S)=\chi(T)=i$ and the distance between $S$ and $T$ is at least four in $H_k$, then the distance
between $S^1$ and $T^2$ is at least $5$ in $F_k$, which contradicts that $E_i$ has diameter at most $4$. Indeed, in any path from
$S^1$ to $T^2$, one of the edges between the two copies of $H_k$ must be used, as well as at least four edges inside copies of
$H_k$. Thus this coloring satisfies conditions of Lemma \ref{lastlemmamindegreediam4} and therefore $\chi$ uses at least $\ell \geq
2k+2$ colors. This
completes the proof of Lemma \ref{lemmalastfordiam4} and of Theorem \ref{largemindegreediam4}.\qed

We next show how to cover all the edges of a graph of large minimum degree by a small number of subgraphs each of diameter at most
$5$.
Recall from Section \ref{sectdiam3} that $G_d(v,w)$ is the induced subgraph of $G$ consisting of all vertices on a walk from $v$ to
$w$ of length at most $d$. Lemma \ref{diamatmostd} states that $G_d(v,w)$ has diameter at most $d$. As in Section \ref{sectdiam4},
$N_r(v)$ denotes all the vertices of $G$ within distance at most $r$ from $v$ and $G_r(v)$ is
a induced subgraph of $G$ with vertex set $N_r(v)$.

\begin{lemma}
We have $Q(n,\epsilon,5) < \epsilon^{-2}$. That is, every graph $G$ with $n$ vertices and minimum degree at least $\epsilon n$ has
an edge covering $E=E_1 \cup \ldots \cup E_{\ell}$ with each $E_i$ having diameter at most $5$ and $\ell < \epsilon^{-2}$.
\end{lemma}
\begin{proof}
Let $\{v_1,\ldots,v_t\}$ be a maximal set of vertices in $G$ of distance more than $2$ apart from each other. By definition,
every vertex of $G$ has distance at most $2$ from one of these vertices. It follows that every edge of $G$ has both its vertices
in some $N_2(v_i)$ or one vertex in a set $N_2(v_i)$ and the other vertex in a different set $N_2(v_j)$.
Note that in the latter case the edge lies in a walk of length at most $5$ from $v_i$ to $v_j$.
For each pair $v_i,v_j$ of distance at most $5$, we will use the subgraph $G_5(v_i,v_j)$. Also, we will use the subgraph
$G_{2}(v_i)$ for each $i$.
The number $\ell$ of subgraphs we use is at most $t+{t \choose 2} \leq t^2$ and every edge is in at least one of these diameter at
most $5$ subgraphs.
 Since $N_1(v_1)$, $\ldots,$ $N_1(v_t)$ are disjoint sets of vertices, each of size at least $1+\epsilon n$,
then $t < \epsilon^{-1}$ and hence $\ell < \epsilon^{-2}$. \end{proof}

We now present a lower bound for $Q(n,\epsilon,5)$.

\begin{theorem} \label{lowerforneps5}
There are positive constants $c$ and $c'$ such that $Q(n,\epsilon,5)>c\epsilon^{-2}$ for $\epsilon>c'(\log n)^{-1/3}$.
\end{theorem}

Let $k$ be minimum positive integer such that $d:=n/a_k \leq 2^{-4}\epsilon^{-1}$ with $a_k=4{4k \choose k}$. The assumption
$\epsilon>c'(\log n)^{-1/3}$ in Theorem \ref{lowerforneps5} and the fact that $a_k$ grows exponentially in $k$ implies that
$k \geq d^3$ for some appropriately chosen constant $c'$. Since $a_{k+1} \leq 10a_k$, we have $d \geq 10 \cdot 2^{-4}\epsilon^{-1}$.
We will construct a graph $G$ on $n$ vertices with
minimum degree at least $\epsilon n$ such that any covering of the edges of $G$ by subgraphs of diameter at most $5$ uses at least
$d^2 \geq 2^{-15}\epsilon^{-2}$ subgraphs, which implies Theorem \ref{lowerforneps5}.

Let $H_k$ be the bipartite graph as defined before Lemma \ref{lastlemmamindegreediam4}. By Lemma \ref{lastlemmamindegreediam4}, $H_k$ is $a_k/16$-regular with  $a_k/2$ vertices, and every coloring of the first vertex class of $H_k$ such that every pair of vertices of the same color have distance less than $4$ uses at least $2k+2$ colors. Let $F$ be the graph consisting of $d$ disjoint copies of $H_k$ with no edges between them. The graph $F$ is bipartite, and we call the union of the $d$ copies of the first vertex class of $H_k$ {\it the first vertex class of $F$}, and the
remaining vertices {\it the second vertex class of $F$}. Let $G$ be the graph consisting of two disjoint copies $F_1$, $F_2$ of $F$, with a
certain matching between the first vertex class of $F_1$ and the first vertex class of $F_2$, which we define in the next paragraph.

Recall that the first vertex class of $H_k$ is the vertex set of the Kneser graph KG$(4k,k)$, and a pair of vertices in the first vertex class of $H_k$ have distance less than $4$ if and only if they are not adjacent in KG$(4k,k)$. Let ${[4k] \choose k} = U_1 \cup \ldots \cup U_d$ be a partition of the vertex set of the Kneser graph KG$(4k,k)$ into $d$ sets such that for $1 \leq i \leq d$  the induced subgraph of KG$(4k,k)$ with vertex set $U_i$ has chromatic number at least $\lfloor 2k/d \rfloor \geq 2d^2$. Such a partition exists since for {\it any} partition $\chi(\mbox{KG}(4k,k))=p_1 + \cdots + p_d$ of the chromatic number of a graph KG$(4k,k)$ into nonnegative integers, there is a partition $U_1 \cup \ldots \cup U_d$ of the vertex set into subsets such that the induced subgraph of KG$(4k,k)$ with vertex set $U_i$ has chromatic number $p_i$. Indeed, we can take $U_i$ to be the union of $p_i$ color classes in a proper coloring of KG$(4k,k)$ with $\chi(\mbox{KG}(4k,k))$ colors. For $j \in
\{1,2\}$, $1 \leq i \leq d$, and $S \subset [4k]$ with $|S|=k$, let $S_{i,j}$ be the copy of vertex $S$ in the $i$th copy of $H_k$ in $F_j$.
Suppose that $S \in U_b$, then the matching between the first vertex class of $F_1$ and the first vertex class of $F_2$ is defined by $S_{i,1}$ is adjacent to $S_{i+b,2}$, where $i+b$ is taken modulo $d$. We denote by $A_{i,j}$ and $B_{i,j}$ the first vertex class and the second vertex class, respectively, of the $i$th copy of $H_k$ in $F_j$.

By definition, graph $G$ has $2d \cdot a_k/2=da_k=n$ vertices and minimum degree at least $a_k/16 \geq \epsilon n$. Hence, Theorem \ref{lowerforneps5} follows from the next lemma.

\begin{lemma}\label{mainlemmafordiam5lb}
Any edge covering of $G$ by diameter at most $5$ subgraphs uses at least $d^2$ subgraphs.
\end{lemma}

Before proving Lemma \ref{mainlemmafordiam5lb}, we first establish some properties of distances between vertices in $G$.

\begin{claim}\label{firstclaimfor5}
If $v \in A_{i,j}$ and $v' \in A_{i',j}$ with $i \not = i'$, then $v$ and $v'$ have distance at least $4$ in $G$.
\end{claim}

Let $v \in A_{i,j}$ and $v' \in A_{i',j}$ be the closest pair of vertices in these two sets. To verify the claim, let
$v=v_1,v_2,\ldots,v_t=v'$ be a shortest path from $v$ to $v'$. The vertex $v_2$, being a neighbor of $v$, must be in $B_{i,j}$ or in
the first vertex class of $F_{3-j}$. If $v_2 \in B_{i,j}$, then $v_3 \in A_{i,j}$ and we could instead start from $v_3$ and get to
$v'$ in a shorter path, a contradiction. Therefore, $v_2$ is in the first vertex class of $F_{3-j}$. Similarly, we must have
$v_{t-1}$ is in the first vertex class of $F_{3-j}$. We have $v_2 \not = v_{t-1}$ since the edge between $v$ and $v_2$ is one of the
edges of the matching, and the edge between $v_{t-1}$ and $v'$ is one of the edges of the matching, which would otherwise imply that
$v_2$ is in two edges of a matching, a contradiction. We have $v_3 \not = v$ as otherwise we would get a shorter path from $v$ to
$v'$. This implies that $v_3$, being a neighbor of $v_2$, lies in the second vertex class of $F_{3-j}$. Hence $t-1 \geq 4$ and the
distance between $v$ and $v'$ is at least $4$.

\begin{claim}\label{secondclaimfor5}
If $v \in B_{i,j}$ and $w \in B_{i',j}$ with $i \not = i'$, then $v$ and $w$ have distance at least $6$ in $G$.
\end{claim}

Indeed, consider a shortest path from $v$ to $w$. The second vertex of this path is in $A_{i,j}$ as all the neighbors of $v$ lie in $A_{i,j}$,
and the second to last vertex of this path is in $A_{i',j}$ as it is adjacent to $w$. By Claim \ref{firstclaimfor5}, any vertex in $A_{i,j}$ has distance at least $4$
from any vertex in $A_{i',j}$, therefore it follows that $v$ and $w$ have distance at least $6$.

\noindent {\bf Proof of Lemma \ref{mainlemmafordiam5lb}:} Suppose for contradiction that there are $r<d^2$ subgraphs $G_1,\ldots,G_r$ of $G$ each of diameter at most $5$ which cover the edges of $G$. It follows from Claim \ref{secondclaimfor5} that any diameter at most $5$ subgraph of $G$ cannot contain a vertex in $B_{i,j}$ and also a vertex in $B_{i',j}$ with $i \not = i'$. Hence, for each $h$, $1 \leq h \leq r$, there is at most one pair $(i,i')$ such that $G_h$ contains both a vertex of $B_{i,1}$ and a vertex of $B_{i',2}$. Since $r<d^2$, the pigeonhole principle implies that there is a pair $(i,i')$ such that no $G_h$ contains a vertex of $B_{i,1}$ together with a vertex of $B_{i',2}$. Fix such a pair $(i, i')$. Let $V_h$ denote the collection of all sets $S \in {[4k] \choose k}$ such that vertices $S_{i,1}$ and $S_{i',2}$ form an edge of the matching that belongs to $G_h$. By definition of such an edge, the set $S$ is in $U_{i'-i}$, where the subscript is taken modulo $d$. Recall that the induced subgraph of KG$(4k,k)$ with vertex set $U_{i'-i}$ has chromatic number at least $2d^2$. We use the following claim.

\begin{claim}\label{fourthclaimfor5}
If $S,S'$ are vertices in the first vertex class of $H_k$, then the distance between $S_{i,j}$ and $S'_{i,j}$ in $G$ is the distance
between $S$ and $S'$ in $H_k$.
\end{claim}

Indeed, $S_{i,j}$ and $S'_{i,j}$ belong to a copy of $H_k$ in $G$, so their distance in $G$ is at most the distance of $S$ and $S'$ in $H_k$. In the other direction, it is easy to see that any path in $G$ from $S_{i,j}$ to $S'_{i,j}$ can be projected onto a path which is not longer from $S$ to $S'$ in $H_k$, which verifies Claim \ref{fourthclaimfor5}.

The next claim completes the proof of Lemma \ref{mainlemmafordiam5lb}. Indeed, since $G_1,\ldots,G_r$ cover the edges of $G$, we must have
$U_{i'-i}=V_1 \cup \ldots \cup V_r$. Claim \ref{thirdclaimfor5} implies that each $V_h$ forms an independent set in KG$(4k,k)$ and hence $r \geq 2d^2$.

\begin{claim}\label{thirdclaimfor5}
Each pair of vertices in $V_h$ have distance less than $4$ in $H_k$, i.e., $V_h$ forms an independent set in KG$(4k,k)$.
\end{claim}

To prove Claim \ref{thirdclaimfor5}, suppose for contradiction that $S,S' \in V_h$ have distance at least $4$ in $H_k$. Without loss of generality, suppose $G_h$ contains no vertex in $B_{i,1}$ (the other case in which $G_h$ contains no vertex in $B_{i',2}$ can be treated similarly). We claim that  the distance between $S_{i,1}$ and $S'_{i,1}$ in $G_h$ is at least $6$. Indeed, the only vertex adjacent to $S_{i,1}$ in $G_h$ is $S_{i',2}$ and the only vertex adjacent to $S'_{i,1}$ in $G_h$ is $S'_{i',2}$. Since $S$ and $S'$ have distance at least $4$ in $H_k$, by Claim \ref{fourthclaimfor5}, $S_{i',2}$ and $S'_{i',2}$ have distance at least $4$ in $G$ and hence also in $G_h$. Therefore, $S_{i,1}$ and $S_{i',2}$ have distance at least $6$ in $G_h$, contradicting $G_h$ has diameter at most $5$. \qed

We next show that $Q(n,\epsilon,6)=\Theta(\epsilon^{-1})$ for $\epsilon \geq 1/n$. The lower bound follows by considering a disjoint
union of cliques, each with at least $\epsilon n+1$ vertices. The upper bound also has a simple proof.

\begin{lemma}
Every graph $G$ with $n$ vertices and minimum degree at least $\epsilon n$ has an edge covering $E=E_1 \cup \ldots \cup E_{\ell}$
with each $E_i$ having diameter at most $6$ and $\ell < \epsilon^{-1}$.
\end{lemma}
\begin{proof}
Let $\{v_1,\ldots,v_{\ell}\}$ be a maximal set of vertices in $G$ of distance more than $2$ apart from each other.
Then every vertex of $G$ has distance at most $2$ from one of these vertices and therefore every edge of $G$ has both of its endpoints within distance at most $3$ from some $v_i$. This implies that the $\ell$ subgraphs $G_3(v_1), \ldots, G_3(v_{\ell})$
cover all the edges of $G$ and each of them  has diameter at most $6$.
Since $N_1(v_1)$, $\ldots,$ $N_1(v_{\ell})$ are disjoint sets of vertices of size at least $1+\epsilon n$, then $\ell < \epsilon^{-1}$.
\end{proof}

\section{Decomposing hypergraphs into low diameter subgraphs}\label{hypergraphsection}

We start this section by showing how to decompose the edge set of a $k$-uniform hypergraph on $n$ vertices, apart from at most
$\epsilon n^k$ edges, into a small number of subhypergraphs of diameter at most $3$. This will establish
an upper bound on $P_k(n,\epsilon,3)$. We then present two constructions giving lower bounds on
$P_k(n,\epsilon,d)$.

Let $H^k$ denote the following $k$-uniform hypergraph on $2k$ vertices with $k+2$ edges. Its vertex set consist of two disjoint
$k$-sets $V=\{v_i\}_{i=1}^k$ and $W=\{w_i\}_{i=1}^k$. The sets $V$ and $W$ are edges of $H^k$. For $1 \leq i \leq k$, vertex $w_i$
together with all vertices $v_j, j \not = i$ also form an edge of $H^k$.

Let $G$ be a $k$-uniform hypergraph and let $e=\{v_1,\ldots,v_k\}$ be a fixed edge of $G$.
Consider all the vertices of $G$ which are contained in some edge of $G$ which intersects $e$ in at least
$k-1$ vertices. Let $G(e)$ be the subhypergraph of $G$ induced by this set. Since the intersection of $e$ with itself has size $k$, by definition, all the vertices of $e$ are in $G(e)$. Moreover $e$ is an edge of $G(e)$ as well.

\begin{lemma}\label{hypdiam3}
For each edge $e$ of a $k$-uniform hypergraph $G$, the diameter of $G(e)$ is at most $3$.
\end{lemma}

\begin{proof} Suppose $a,b$ are distinct vertices of $G(e)$. Then there are two indices $1 \leq i,j \leq k$ such that
$\{a\}\cup (e \setminus \{v_i\})$ and $\{b\}\cup (e \setminus \{v_j\})$ are both edges of $G(e)$.
If $i=j$, then the sequence $a$, followed by all elements of $e \setminus v_i$, followed by $b$ is a path of length $2$. If $i
\not
=j$, and $a=v_i$ or $b=v_j$, then $a$ and $b$ are in an edge of $G(e)$. If $i \not =j$ and neither $a = v_i$ nor
$b=v_j$, then the sequence with first element $a$, followed by $v_j$, followed by all vertices of $e\setminus \{v_i,v_j\}$,
followed by $v_i$ and finally by $b$ is a path of length $3$. In any case, the distance from $a$ to $b$, and hence the diameter
of $G(e)$, is at most $3$. \end{proof}

The number of edges of $G(e)$ is at least the number of copies of $H^k$ in $G$ for which the image of
set $V$ is fixed as $e$. Indeed, if two disjoint edges $e$ and $f$ (together with some other edges of $G$) form a copy of $H^k$ with
$f$ being the image of $W$, then each vertex in $f$ is in
$G(e)$, and so $f$ is an edge of $G(e)$.

The {\it edge density} of a $k$-uniform hypergraph is the fraction of subsets of vertices of size $k$ which are edges.
Let $K(t;k)$ denote the complete $k$-partite $k$-uniform hypergraph with parts of size $t$, whose edges are all the $k$-sets which
intersect every part in one vertex. This hypergraph has $kt$ vertices and $t^k$ edges.
The following well known lemma is proved by a straightforward counting argument and induction on the uniformity $k$ (see \cite{ErSi}).

\begin{lemma} \label{lemsidcomplete} Fix positive integers $k$ and $t$.
If $G$ is a $k$-uniform hypergraph with $n$ vertices and edge density $\epsilon$ with $\epsilon \gg n^{-t^{1-k}}$, then $G$
contains $\Omega(\epsilon^{t^k}n^{kt})$ labeled copies of $K(t;k)$.
\end{lemma}

Let $G$ be a $k$-uniform hypergraph with $n$ vertices and edge density $\epsilon$. Since $H^k$ is a subhypergraph of
$K(2;k)$, the above lemma implies that $G$ contains $\Omega \big(\epsilon^{2^k}n^{2k}\big)$ labeled copies of
$K(2;k)$ and hence of $H^k$. Therefore, there is an edge $e$ for which $G(e)$ contains at least
$\Omega\big(\epsilon^{2^k}n^{2k}/\epsilon{n \choose k}k!\big)=\Omega\big(\epsilon^{1-2^k}n^k\big)$ edges. By Lemma
\ref{hypdiam3}, this subhypergraph has diameter at most $3$. Now, as we already did in the previous sections, we can
use the above fact to pull out from $G$ subhypergraphs of diameter $3$ with many edges until there are at
most $\epsilon n^k$ edges left. Using a very similar computation as in the proofs of
Theorems \ref{thmfor3} and \ref{thmfor4}, we obtain Theorem \ref{mainhyper}
that $P_k(n,\epsilon,3) =O(\epsilon^{2-2^k})$ for
$\epsilon \gg n^{-2^{1-k}}$.

We next give a lower bound on $P_k(n,\epsilon,3)$ which we conjecture is tight apart from the constant factor.

\begin{theorem} \label{hyplowerconst}
We have $P_k(n,\epsilon,3) \geq c_k\epsilon^{-k}$ for $C_k'n^{-1/2} \leq \epsilon \leq C_k$, where $c_k$, $C_k$ and $C_k'$ are
positive constants depending only on $k$.
\end{theorem}

We prove this theorem by showing that there is a dense $k$-uniform hypergraph with no large subhypergraph of diameter at most $3$.
More precisely, the next lemma shows that there is a hypergraph with $n$ vertices, at least $2\epsilon n^k$ edges, and every
subhypergraph of diameter at most $3$ has at most $c^{-1}_k\epsilon^{k+1}n^k$ edges. Hence $P_k(n,\epsilon,3) \geq (\epsilon
n^k)/(c^{-1}_k\epsilon^{k+1}n^k) = c_k\epsilon^{-k}$.

\begin{lemma}
For each integer $k \geq 2$, there are positive constants $c_k, C_k$ and $C_k'$ such that the following holds.
For all sufficiently large $n$ and $\epsilon$ satisfying $C_k' n^{-1/2} \leq \epsilon \leq C_k$,
there is a hypergraph $H$ on at most $n$ vertices with at least $2\epsilon n^{k}$ edges such
that every subhypergraph of $H$ with diameter at most $3$ has
at most $c^{-1}_k\epsilon^{k+1}n^k$ edges.
\end{lemma}

\begin{proof} The proof is by induction on $k$. We have already established the base case $k=2$ of this lemma in the proof of
Theorem \ref{diam3lwrb}. Let $C_{k+1} = 2^{-k-3} C_k$ and let $\epsilon \leq C_{k+1}$. Fix $\delta=2^{k+3}\epsilon$ and $N=n/2$. Our
induction hypothesis implies that there is a $k$-uniform hypergraph $G$ with $N$ vertices and $2\delta N^k$ edges such that every diameter at most $3$ subhypergraph of $G$ has at most $c^{-1}_k\delta^{k+1}N^k$ edges. So $G$ has edge density $\alpha = 2\delta N^k/{N \choose k}\leq 4k!\delta$.

Let $G_1,\ldots,G_t$ be $t=\frac{2}{\alpha}$ random copies of $G$ on the same vertex set $[N]$, where $G_i$ is formed by considering
a random bijection of $V(G)$ to $[N]$ picked independently of all the other $G_j$. The probability that a given $k$-tuple contained
in $[N]$ is an edge of at least one of these $t$ copies of $G$ is $$1-(1-\alpha)^t \geq 1-e^{-\alpha t} = 1-e^{-2} \geq 1/2.$$ By
linearity of expectation, the expected number of edges which are contained in at least one of the $t$ copies of $G$ is at least
$\frac{1}{2}{N \choose k}$. So we may pick $G_1,\ldots,G_t$ such that at least $\frac{1}{2}{N \choose k}$ of the $k$-tuples
contained in $[N]$ are in at least one of these $t$ copies of $G$.

If an edge $e$ is in multiple copies of $G$, arbitrarily delete it from all $G_i$ that contain it except one. Let $G_i'$ be the
resulting subhypergraph of $G_i$. Introduce new vertex sets $V_1,\ldots,V_t$, each of size $N/t$, and define a $(k+1)$-uniform
hypergraph $H$ on $[N] \cup V_1 \cup \ldots \cup V_t$ as follows. The edges are those $(k+1)$-sets which for some $i$ contain an
edge of $G_i'$ together with a vertex of $V_i$. The number of vertices of $H$ is $2N=n$ and the number of edges is at
least $$(N/t)\frac{1}{2}{N \choose k}=\frac{\alpha N}{4}{N \choose k} \geq \frac{1}{2} \delta N^{k+1} = 2\epsilon n^{k+1}.$$ Note
that every edge of $H$ has exactly one vertex in some $V_i$. Let $H'$ be a diameter at most $3$ subhypergraph of $H$. Since the
$G_i'$ have distinct edges, the vertex set of any strongly connected subhypergraph of $H$, and hence of $H'$, intersects at
most one $V_i$. Therefore, there is an $i$ such that $V(H') \subset V_i \cup [N]$ and $E(H') \subset V_i \times E(G'_i)$. Since each
edge of $H'$ has exactly one vertex in $V_i$, note that if we delete from any tight path in $H'$ all the vertices of $V_i$ we obtain
a tight path in $G_i'$. Therefore the subhypergraph of $G_i'$ whose edges are those which are subsets of edges
of $H'$ also has diameter at most $3$. Since $G_i'$ is a subhypergraph of a copy of $G$, any diameter $3$ subhypergraph of $G_i'$
has at most $c^{-1}_k\delta^{k+1}N^k$ edges. Thus, using that $n=2N$ and $1/t=\alpha/2\leq 2k!\delta$, we obtain that $H'$ has at most
$$|V_i|c^{-1}_k\delta^{k+1}N^k = (N/t)c^{-1}_k\delta^{k+1}N^k \leq
2k!c^{-1}_k\delta^{k+2}N^{k+1}=2^{k^2+4k+6}k!c^{-1}_k\epsilon^{k+2}n^{k+1}$$ edges. Letting $c_{k+1}=2^{-k^2-4k-6}c_k/k!$
completes the
proof by induction. \end{proof}

Essentially the same proof gives the following lemma. The only difference is the base case $k=2$ of the induction is given by taking
$H$ to be a disjoint union of cliques of equal size. Call a $k$-uniform hypergraph {\it strongly connected} if there is a tight path
between any two vertices in the hypergraph, i.e., the diameter of the hypergraph is finite.

\begin{lemma}
For each integer $k \geq 2$, there are constants $c_k, C_k$ and $C_k'$ such that the following holds.
For all sufficiently large $n$ and $\epsilon$ satisfying $C_k' n^{-1} \leq \epsilon \leq C_k$,
there is a hypergraph $H$ on $n$ vertices with $2\epsilon n^{k}$ edges
such that every strongly connected subhypergraph of $H$
has at most $c^{-1}_k\epsilon^{k}n^k$ edges.
\end{lemma}

The hypergraph $H$ in the previous lemma demonstrates the following corollary.

\begin{corollary}\label{lastcork} $P_k(n,\epsilon,d) \geq c_k \epsilon^{1-k}$ for $C_k'n^{-1} \leq \epsilon \leq C_k$ and
any $d$.
\end{corollary}

\section{Concluding remarks}
\label{concludingremarks}

In the previous section, we establish both upper and lower bounds on $P_k(n,\epsilon,d)$. We think the lower bounds in Theorem
\ref{hyplowerconst} is best possible up to a constant factor, i.e.,
all but at most $\epsilon n^k$ edges of every $n$-vertex $k$-uniform hypergraph can be
partitioned into $O(\epsilon^{-k})$  subhypergraphs of diameter $3$. We also believe that
Corollary \ref{lastcork} is tight and for each integer $k \geq 3$ there is another integer $d(k)$ such that $P_k(n,\epsilon,d(k)) =O(\epsilon^{1-k})$.
Note that Theorem \ref{main1} shows that both our conjectures hold for graphs ($k=2$).

Improving our upper bound on $P_k(n,\epsilon,3)$ would also be interesting. One possible way to do so
is to show that there are many copies of hypergraph $H^k$ in every $k$-uniform
hypergraph on $n$ vertices with edge density $\epsilon$. This problem is closely related to the well know conjectures of Simonovits
\cite{Sim} and Sidorenko \cite{Si3}, which suggest that for any bipartite graph $H$, the number of its copies in any graph $G$ on
$n$ vertices and edge density $\epsilon$ ($\epsilon>n^{-\gamma(H)}$) is asymptotically at least the same as in the $n$-vertex random
graph with the same edge density. So far it is known only in very special cases, i.e., for complete bipartite graphs,
trees, even cycles (see \cite{Si3}), and recently for cubes \cite{Ha}. It is tempting to conjecture that an analogous statement holds for
$k$-uniform $k$-partite hypergraphs ($k\geq 3)$. A $k$-uniform hypergraph is {\it $k$-partite} if there is a partition of the vertex
set into $k$ parts such that each edge has exactly one vertex in every part. Simonovits-Sidorenko conjecture for hypergraphs would
say that for any $k$-partite $k$-uniform hypergraph $H$, the number of its copies in any $k$-uniform hypergraph on $n$ vertices with
edge density $\epsilon$ ($\epsilon>n^{-\gamma(H)}$) is asymptotically at least the same as in the random $n$-vertex $k$-uniform
hypergraph with edge density $\epsilon$. It holds for {\it complete} $k$-partite $k$-uniform $H$ (Lemma \ref{lemsidcomplete} is
essentially the case when all parts of $H$ have equal size). However, as shown by Sidorenko \cite{Si3}, this is false in general.
Nevertheless, it is still an intriguing open problem to accurately estimate the minimum number of copies of a fixed hypergraph $H$ that
have to appear in every $k$-uniform hypergraph on $n$ vertices with edge density $\epsilon$.

\end{document}